\documentclass[12pt]{amsart}
\usepackage{amssymb,amsmath,amsthm,comment}
\usepackage{pdfpages,graphicx} 
\usepackage[section]{placeins} % figures should be displayed in same section
\usepackage{wrapfig}
\usepackage{float}
\usepackage{lineno}
\usepackage{youngtab}

\usepackage{setspace}
\newcommand{\shrinkmargins}[1]{
  \addtolength{\textheight}{#1\topmargin}
  \addtolength{\textheight}{#1\topmargin}
  \addtolength{\textwidth}{#1\oddsidemargin}
  \addtolength{\textwidth}{#1\evensidemargin}
  \addtolength{\topmargin}{-#1\topmargin}
  \addtolength{\oddsidemargin}{-#1\oddsidemargin}
  \addtolength{\evensidemargin}{-#1\evensidemargin}
  }
\shrinkmargins{0.5}

\newtheorem{theorem}{Theorem}
\newtheorem{lemma}[theorem]{Lemma}

\newtheorem{corollary}[theorem]{Corollary}
\newtheorem*{theorem*}{Theorem}

{Claim}

\theoremstyle{definition}
\newtheorem*{definition}{Definition}

\theoremstyle{remark}

\newtheorem*{remarks}{{\bf Remarks}}

\newtheorem*{examples}{Examples}
\numberwithin{theorem}{section} \numberwithin{equation}{section}

\usepackage{caption}
\usepackage{multirow}

\def\func#1{\mathop{\rm #1}}%
\begin{document}
\title[Tur\'an Inequalities]{Tur\'an inequalities from Chebyshev to Laguerre polynomials}
\author{Bernhard Heim }
\address{Lehrstuhl A f\"{u}r Mathematik, RWTH Aachen University, 52056 Aachen, Germany}
\email{bernhard.heim@rwth-aachen.de}
\author{Markus Neuhauser}
\address{Kutaisi International University, 5/7, Youth Avenue,  Kutaisi, 4600 Georgia}
\address{Lehrstuhl A f\"{u}r Mathematik, RWTH Aachen University, 52056 Aachen, Germany}
\email{markus.neuhauser@kiu.edu.ge}
\author{Robert Tr\"oger}
\email{robert@silva-troeger.de}
%%\linenumbers
\begin{abstract}
Let $g$ and $h$ be
real-valued arithmetic functions, positive and normalized. 
Specific choices within the following general scheme of recursively defined polynomials
\begin{equation*}
P_n^{g,h}(x):= \frac{x}{h(n)} \sum_{k=1}^{n} g(k) \, P_{n-k}^{g,h}(x),
\end{equation*}
with initial value $P_{0}^{g,h}(x)=1$
encode information about several classical, widely
studied polynomials. This includes
Chebyshev polynomials of the second kind, associated Laguerre polynomials,
and the Nekrasov--Okounkov polynomials.
In this paper we prove that for $g=\func{id}$ and fixed $h$ we obtain
orthogonal polynomial sequences for positive definite functionals.
Let $h(n)=n^s$ with $0 \leq s \leq 1 $.
Then the sequence satisfies
Tur\'an inequalities for $x \geq 0$. 
\end{abstract}
%%%%%%%%%%%%%%%%%%%%%%%%%%%%%%%%%%%%%%%%%%%%%%%%%%%%%%%%%%%%%%%%%%%%%%%%%%%%%%%%%%%%%%%%%%%%%%%%%%%%%%%%%%
%%%%%%%%%%%%%%%%%%%%%%%%%%%%%%%%%%%%%%%%%%%%%%%%%%%%%%%%%%%%%%%%%%%%%%%%%%%%%%%%%%%%%%%%%%%%%%%%%%%%%%%%%%
%%%%%%%%%%%%%%%%%%%%%%%%%%%%%%%%%%%%%%%%%%%%%%%%%%%%%%%%%%%%%%%%%%%%%%%%%%%%%%%%%%%%%%%%%%%%%%%%%%%%%%%%%%
%%%%%%%%%%%%%%%%%%%%%%%%%%%%%%%%%%%%%%%%%%%%%%%%%%%%%%%%%%%%%%%%%%%%%%%%%%%%%%%%%%%%%%%%%%%%%%%%%%%%%%%%%%
\maketitle
\newpage
\section{Introduction and main results}
Orthogonal polynomials play an important role in mathematics and applied sciences.
The theory of orthogonal polynomials and their main properties 
is
well documented \cite{Fr71,Sze75,Ne79, VA08, Ch11}. In this
paper we define for every positive normalized arithmetic function $h(n)$ an orthogonal polynomial sequence with positive definite
moment functionals. We propose certain properties on $h(n)$ which imply Tur\'an inequalities.
We show that $h(n)=h_s(n):= n^s$ for $s \in [0
,1]$ provides complete results.

In 1948,
in a seminal work, Szeg\H{o} \cite{Sze48} called attention to a discovery
by Tur\'an \cite{Tu50}.
Studying the zeros of Legendre polynomials $P_n(x)$ of degree $n$, Tur\'an
proved and utilized a new type of inequalities for orthogonal polynomials:
\begin{equation*}
T_n(x):=  P_n(x)^2 - P_{n-1}(x) \, P_{n+1}(x), \qquad x \in [-1,1].
\end{equation*}
Szeg\H{o} presented several proofs and showed that Tur\'an inequalities are also satisfied
by other orthogonal polynomials, including ultraspherical, associated Laguerre,
and Hermite polynomials.
Proofs include transcendental functions and Jensen polynomials. 

This initiated research in several
directions \cite{KS60}, including higher Tur\'an inequalities
capturing the conjectured hyperbolic property of Jensen polynomials
associated with the Riemann zeta
function \cite{Cs11,CJW19, GORZ19}. We also refer to the open session report
by van Assche \cite{VA08},
where Nevai emphasized also the importance of
studying Tur\'an inequalities. The way we normalize the polynomials is essential for the Tur\'an
inequality to hold, we refer to the recent results of Szwarc \cite{Szw98, Szw21}. In this paper we
follow a uniform approach.

We define a sequence of polynomials $\{P_{n}^{g,h}(x)\}_n$
associated with
positive valued normalized arithmetic functions $g$ and $h$. Let $P_{0}^{g,h}(x):=1$ and
\begin{equation*}
P_n^{g,h}(x):= \frac{x}{h(n)} \sum_{k=1}^{n} g(k) \, P_{n-k}^{g,h}(x).
\end{equation*}
The Nekrasov--Okounkov polynomials \cite{NO06} are provided for
$g(n)= \sigma(n):= \sum_{d \mid n} d$ and $h(n)= \func{id}(n)=n$.
Much simpler, but still interesting is the case $g(n)=
{n}$, which involves associated
Laguerre polynomials $L_n^{(1)}(x)$ and Chebyshev polynomials $U_n(x)$ of the second kind \cite{Ch11}.
We state our first result.
\begin{theorem} \label{th:1}
Let $g= \func{id}$ and $h$
be arbitrary positive real valued normalized arithmetic functions.
The sequence $\{q_n^h(x)\}_{n=1}^{\infty}$ defined by
\begin{equation*}
q_n^h(x) := \frac{\prod_{k=1}^{n+1} h(k)}{x} \, P_{n+1}^{\func{id},h}(x)
\end{equation*}
is an orthogonal polynomial sequence (OPS) with a positive definite moment functional $\Lambda $.
\end{theorem}
In Section \ref{abschnitt2}
we recall basic properties for the OPS and prove Theorem \ref{th:1}.
This result leads to the following application.
\begin{corollary}\label{cor:1}
The polynomials $P_n^{\func{id},h}(x)$ have simple, real, and non-positive
zeros. 
Moreover, the
zeros of $\frac{1}{x} \, P_n^{\func{id},h}(x)$
are interlacing.
\end{corollary}
This follows from (\cite{Ch11}, Theorem 5.2 and Theorem 5.3). Note that the coefficients of
$P_n^{\func{id}, h}(x)$ are non-negative.

Suppose $h$ is log-concave and satisfies some further properties, then
the polynomials $P_n^{\func{id},h}(x)$ satisfy Tur\'an inequalities.
\begin{theorem}\label{suppose}
Let $h$ be a real-valued normalized arithmetic function with positive values, extended by $h(0):=0$.
Suppose that
\begin{equation*}\label{h:lc}
\Delta_h(n):= h(n)^2 - h(n-1) \, h(n+1) \geq 0 \qquad \text{ for all } n \geq 1.
\end{equation*}
Let $n \geq 1$ and $x \geq 0$. Let 
\begin{equation*}\label{v:root}
v_{n,2}^h(x):= \frac{1}{h(n+1)} \,\, \left( x + 2 \, h(n) + \sqrt{x^2 + 4 \, h(n) \,x + 4 \Delta_h(n)}  \right) .
\end{equation*}
Suppose that for all $n \geq 1$ and $x \geq 0$:
\begin{equation}\label{decreasing}
v_{n,2}^h(x) \geq v_{n+1,2}^h(x).
\end{equation}
Then the polynomials $P_n^{\func{id},h}(x)$ satisfy Tur\'an inequalities for all
$n \geq 2$ and $x \geq 0$:
\begin{equation}\label{Turan}
T_n^h(x):= P_n^{\func{id},h}(x)^2 -  P_{n-1}^{\func{id},h}(x) \, 
P_{n+1}^{\func{id},h}(x) \geq 0.
\end{equation}
\end{theorem}
Let $s \in [0,
1]$. We prove that the arithmetic function $h_s(n):= n^s$ 
satisfies the decreasing property
(\ref{decreasing}). It also leads to  a transformation of the
distribution of zeros of the OPS for $s=0$ to the
distribution of zeros for
$s=1$. This gives an interesting {\it continuous} transformation 
between Chebyshev polynomials of the second kind and associated 
Laguerre polynomials, keeping the OPS property and satisfying Tur\'an inequalities.
This follows from the fact that $P_n^{\func{id},h_0}(x)= x U_{n-1}(\frac{x}{2}+1)$ and 
$P_n^{\func{id},h_1}(x) = \frac{x}{n} L_{n-1}^{(1)}(x)$. We refer to \cite{HNT20} and \cite{HLN19}.
To summarize, we have the following:
\begin{theorem} \label{result}
Let $s\in [0,1]$ and $x \geq 0$. Then,
for all $n \geq 2
$, we have the Tur\'an inequalities
\begin{equation*}
T_n^{h_s}(x):= P_n^{\func{id},h_s}(x)^2 - P_{n-1}^{\func{id},h_s}(x) \,\, P_{n
+1}^{\func{id},h_s}(x) \, \geq \, 0.
\end{equation*}
\end{theorem}
We have displayed the
transformation of the zeros for $n=7$
in Figure \ref{realroots}.
Note for $s=0$ we have plotted the
zeros of $P_7^{\func{id},h_0}(x)$, then approximated
all the
zeros for each $P_7^{\func{id},h_s}(x)$, $ 0 < s < 1$, and finally
attain the
zeros of $P_7^{\func{id},h_1}(x)$.

\begin{minipage}{0.45\textwidth}
%%\includegraphics[width=\textwidth]{./concave_delta_ap1_am1_100_small.png}
% active \includegraphics[width=\textwidth]{./concave_delta_ap1_am1_100_size_4_4.png}
\includegraphics[width=\textwidth]{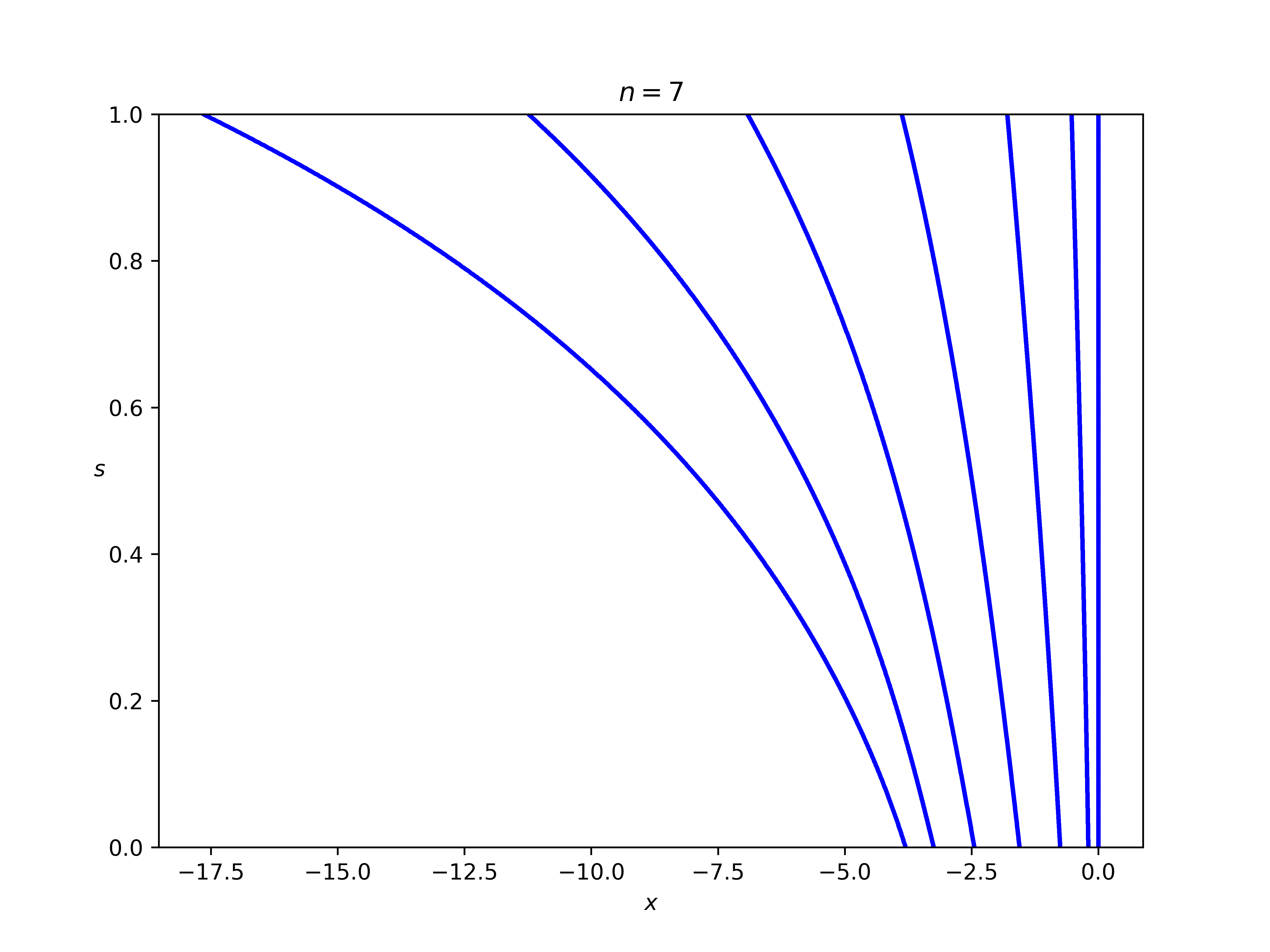}
\end{minipage}
\hfill
\begin{minipage}{0.52\textwidth}
\vspace{1.5cm}
\captionsetup{width=\linewidth}
\captionof{figure}{\label{realroots}
Transformation: \\
zeros
of $P_{7}
^{\func{id},h_s}(x)$ for \\ $ 0 \leq s \leq 1$.
%Zeros of $\Delta_{a+1, a-1}(x)$ with real part larger zero for $ 1 \leq a \leq 100$.
%Blue labels the real zeros and red the imaginary zeros.
}
\end{minipage}

%%%%%%%%%%%%%%%%%%%%%%%%%%%%%%%%%%%%%%%%%%%%%%%%%%%%%%%%%%%%%%%%%%%%%%%%%%%%%%%%%%%%%%%%%%%%%%%%%%%%%%
%%%%%%%%%%%%%%%%%%%%%%%%%%%%%%%%%%%%%%%%%%%%%%%%%%%%%%%%%%%%%%%%%%%%%%%%%%%%%%%%%%%%%%%%%%%%%%%%%%%%%%
%%%%%%%%%%%%%%%%%%%%%%%%%%%%%%%%%%%%%%%%%%%%%%%%%               OPS
\section{\label{abschnitt2}OPS and proof of Theorem \ref{th:1}}
We first recall basic definitions and results related to Favard's theorem.
Then we provide a three
term
recurrence relation for $\left\{ P_n^{\func{id},h}\left( x\right) \right\} $ and finally prove Theorem \ref{th:1}.
\subsection{Orthogonal polynomial sequences and Favard's theorem}
Every linear functional $\Lambda $ on the vector space of polynomials $\mathbb{C}[x]$ is given and
uniquely determined by the complex values $\Lambda (x^n)= \mu_n$ on the standard basis.
The sequence $\{\mu_n\}_{n=0}^{\infty}$ is
called moment sequence and $\Lambda $ moment functional.
We closely follow Chihara (\cite{Ch11}, Sections 1.2, 1.3, and 1.4). 
\begin{definition}
A sequence $\{P_n(x)\}_{n=0}^{\infty}$ is called orthogonal polynomial sequence (OPS) with respect to 
a moment functional $\Lambda $ if for all $n,m \in \mathbb{N}_0$:
\begin{itemize}
\item[a)] $P_n(x)$ is a polynomial of degree $n$,
\item[b)] $\Lambda \left( P_m(x) \, P_n(x) \right) = k_n \, \delta_{m,n}, \quad
k_n \neq 0.
$
\end{itemize}
If this is the case we say, $\{P_n(x)\}$ is an OPS for $\Lambda $.
\end{definition}
Let $\left\{ \mu_n\right\} $ be a moment
sequence
and let $\Delta_n$ be the determinant of the Hankel matrix
$(\mu_{i+j})_{i,j=0}^n$. Let $\Lambda $ be a moment functional with moment sequence $\left\{
\mu _{n}
\right\} $.
Then we have (\cite{Ch11}, Theorem 3.1) for moment functionals $\Lambda $:
\begin{equation*}
\text{OPS for }\Lambda  \text{ exists } \Longrightarrow \Delta_n \neq 0, \qquad
n \geq 0.
\end{equation*}

Let $E \subset (-\infty,\infty)$. Let $P(E)$ be the set of all polynomials with real coefficients,
restricted to $E$ non-negative and not the zero polynomial. 
A moment functional is said to be positive definite on $E$
if $\Lambda $ restricted to $P(E)$ is non-negative. The set $E$ is called a supporting set for $\Lambda $.
If $E= (-\infty, \infty)$ we say $\Lambda $ is positive definite, $\Lambda  >0$.

Let $\Lambda $ be positive definite, then one can conclude that the functional moments are real and that
a corresponding OPS of real polynomials exists (\cite{Ch11}, Theorem 3.3). With  a little bit more effort one
obtains the important characterization of positive definite moment functionals (\cite{Ch11}, Theorem 3.4):       %.
\begin{equation*}
\Lambda  > 0 \Leftrightarrow \mu_n \in \mathbb{R} \text{ and } \Delta_n > 0,\quad
n \geq 0.
\end{equation*}
All $\Lambda $ with $\Delta_n \neq 0$, $
n \geq 0
$, are
called quasi-definite. Now we are able to
state an important result relating OPS to three
term recurrence relations.

\begin{theorem}[Favard's theorem]\label{Favard}
Let $\{c_n\}_{n=1}^{\infty}$ and $\left\{ \lambda _{n}\right\} _{n=1}^{\infty}$
be arbitrary sequences of complex numbers.
Let $\{P_n(x)\}_{n=0}^{\infty}$ be defined by the
recurrence relation
\begin{equation}\label{3term}
P_n(x) = (x-c_n) \, P_{n-1}(x) - \lambda_n \, P_{n-2}(x), \quad
n \geq 1,
\end{equation}
with initial values $P_{-1}(x):=0$ and $P_0(x):=1$.
Then there exists a unique moment functional $\Lambda $, such that $\Lambda (1)= \lambda_1$ and
such that for all
non-negative integers $m,n$ with $m \neq n$, we have $\Lambda (P_m(x) \, P_n(x))=0$. Further,
\begin{itemize}
\item  $\Lambda $ is quasi-definite if and only if
$\lambda_n \neq 0$ for all $n$,

\item  $\Lambda $ is positive definite if and only if
$c_n$ is real and $\lambda_n > 0$ for all $n$.
\end{itemize}
\end{theorem}

\begin{remarks} 
\
\begin{itemize}
\item[1)]Theorem \ref{Favard} was first announced in 1935 by Favard, and independently
discovered by Shohat and Natanson at the same time. It is also contained implicitly in some work
by Stieltjes (\cite{Ch11}, Chapter I, Section 4).
\item[2)] Let $\Lambda $ be quasi-definite. Then there exists $c_n$ and $\lambda_n \neq 0$, such that
the monic OPS of $\Lambda $ satisfies (\ref{3term}). If $\Lambda $ is positive definite, then $c_n$ is real and
$\lambda_n >0$ (converse of Favard's theorem).
\item[3)]An OPS defined by (\ref{3term}) is independent from $\lambda_1$. Usually one puts $\lambda_1 = \Lambda (1)= \mu_1$.
\end{itemize}

\end{remarks}
\subsubsection{Normalization}
We have the following very useful property.
Suppose $\{p_n(x)\}_n$ satisfies the three
term relation (\ref{3term}). Let $a,b \in \mathbb{C}$ with $a\neq 0$.
Then $\{\tilde{p}_n(x)\}_n$, which is defined
as $\tilde{p}_n(x):= a^{-n} \, p_n(ax +b)$, satisfies  (\ref{3term}) with
\begin{equation*}
\tilde{p}_n(x) = \left( x - \frac{c_n - b}{a}\right) \,\tilde{p}_{n-1}(x)- \frac{\lambda_n}{a^2} \tilde{p}_{n-2}(x).
\end{equation*}
Moreover, let $\{p_n(x)\}_n$ be an OPS with moments $\{\mu_n\}_n$. Then $\tilde{p}_n(x)$ is an OPS with
moments
 \begin{equation*}
 \tilde{\mu}_n =  a^{-n} \sum_{k=0}^n \binom{n}{k} \, (-b)^{n-k} \, \mu_k.
 \end{equation*}

Polynomials satisfying (\ref{3term}) are monic. Suppose $\{\hat{p}_n(x)\}$ is an OPS with positive definite moment functional
$\Lambda $ satisfying (\ref{3term}). Then $\Lambda (\hat{p}_n(x)\, \hat{p}_m(x)) = k_n \, \delta_{n,m}$, $
k_n \neq 0
$.
To obtain $k_n=1$ one can drop the monicity property and study a modified three
term recurrence relation.

Let $\{a_n\}_{n=0}^{\infty}$ be a sequence of positive numbers and $\{b_n\}_{n=1}^{\infty}$ a sequence of real numbers:
\begin{equation*}
x \, p_n(x) = a_{n+1} \, p_{n+1}(x) + b_n \, p_n(x) + a_n \, p_{n-1}(x), \quad
n \geq 0
,
\end{equation*}
with initial values $p_{-1}(x)=0$ and $p_0(x)=1$. Let $\Lambda (p_n(x) p_m(x)) = \delta_{i,j}$.
Then $\hat{p}_n(x) = ( a_1\, a_2 \, \cdots a_n ) \, p_n(x)$, satisfying (\ref{3term}).
This implies
\begin{equation*}
\hat{p}_{n+1}(x)= ( x- b_n) \hat{p}_n(x) - a_n^2 \, \hat{p}_{n-1}(x), \quad
n \geq 0
,
\end{equation*}
with $\hat{p}_{-1}(x)= p_{-1}(x)=0$ and $\hat{p}_{0}(x)= p_{0}(x)=1$. Here
$c_{n} = b_{n-1}$ and $\lambda _{n} = a_{n-1}^{2}$.
\newline
\newline

\subsection*{Applications}

Let $\{U_n(x)\}_n$ be the Chebyshev polynomials of the second kind. The first polynomials
are given by: $U_0(x)=1, U_1(x)= 2x, U_2(x) = 4x^2- 1, U_3(x)= 8 x^3 -4x$. Then $\hat{U}_n(x) = 2^{-n} U_n(x)$ and
\begin{equation*}
\hat{U}_n(x) = x \hat{U}_{n-1}(x) - \frac{1}{4} \hat{U}_{n-2}(x).
\end{equation*}
The Chebyshev polynomials $U_n(x)$ of the second kind satisfy:
\begin{equation*}
U_n(x)= 2x U_{n-1}(x) - U_{n-2}\left( x\right) , \quad
n \geq 1
\end{equation*}
with initial values $U_{-1}(x)=0$ and $U_0(x)=1$. We have $c_n=0$ and $\lambda_n =1$.
This implies that $\{U_n(x)\}$ is an OPS for a positive definite functional $\Lambda $.
Let $p_n(x):= a^{-n} U_n(ax +b)$ for $ a \neq 0$. Then
\begin{equation*}
p_n(x)= \left( x- \frac{-b}{a}\right) \, p_{n-1}(x) - \frac{1}{a^2} p_{n-2}(x), \quad
n \geq 1
,
\end{equation*}
with initial values $p_{-1}(x)=0$ and $p_0(x) = 1$. Thus, $\{p_n(x)\}$ is also OPS for a positive definite $\Lambda $.
Let $a=1/2$ and $b=1$. Then we have $\tilde{c}_n= -2$ and $\tilde{\lambda}_n = 4$ for this OPS.
Let $\alpha >-1$. Then the $\alpha$-associated Laguerre polynomial $L_n^{(\alpha)}(x)$ is defined
by
\begin{equation}\label{associated}
L_n^{(\alpha)}(x) = \sum_{k=0}^n \binom{n + \alpha}{n-k} \frac{(-x)^k}{k!}.
\end{equation}
The polynomials have the leading coefficient $k_n= (-1)^n/n!$. It is known that
\begin{equation*}
n L_n^{(\alpha)}(x) = (2n + \alpha -1-x) \, L_{n-1}^{(\alpha)}(x)  - (n+ \alpha-1)  \,    L_{n-2}^{(\alpha)}(x).
\end{equation*}
Therefore, associated Laguerre polynomials are orthogonal polynomials.
For the corresponding monic polynomials $\hat{L}_n^{(\alpha)}(x) := \frac{1}{k_n}\,  L_n^{(\alpha)}(x)$ we obtain
\begin{equation*}
\hat{L}_n^{(\alpha)}(x) = (x - 2n - \alpha +1) \, \hat{L}_{n-1}^{(\alpha)}(x)  - (n-1)(n+ \alpha-1)  \,    \hat{L}_{n-2}^{(\alpha)}(x).
\end{equation*}
Let $\alpha = 1$. Then we obtain
\begin{equation*}
\hat{L}_n^{(1)}(x) = (x - 2n) \, \hat{L}_{n-1}^{(1)}(x)  - n\, (n-1)  \,    \hat{L}_{n-2}^{(1)}(x),
\end{equation*}
where $c_n= 2n$ and $\lambda_n= (n-1)n$. Moreover let $p_n(x):= (-1)^n \, \hat{L}_n^{(1)}(-x)$.
Then we obtain:
\begin{equation*}
p_n(x)= \left( x + 2n \right) \, p_{n-1}(x) - (n-1)\, n \,  p_{n-2}(x), \quad
n \geq 1
,
\end{equation*}
Let us extend (\ref{associated}) to $\alpha=-1$. Then $\{L_n^{\left( -1\right) }(x)\}$ is not an OPS.
This is still interesting, since $P_n^{\func{id},\func{id}}(x)= L_n^{\left( -1\right) }(x)$. This leads almost
to an OPS, since $L_n^{(-1)}(x) = \frac{x}{n} L_{n-1}^{(1)}(-x)$.

%%% Chihara S154

\subsection{Three term
recurrence relation for $\{P_n^{\func{id},h}(x)\}_n$}
Let $\{P_n^{\func{id},h}(x)\}_n$ be given. Let $g(n)= \func{id}(n)=n$.
Then we
define $q_{-1}^{h}\left( x\right) :=0$ and
\begin{equation*}
q_n^h(x):= \frac{\prod_{k=1}^{n+1} h(k)}{x} \, P_{n+1}^{
g,h}(x), \qquad
n \geq
0
.
\end{equation*}

\begin{theorem}
Let $n  \geq 1$. Let $h$ be a non-vanishing function on $\mathbb{N}_0$ with $h(0)=0$ and $h(1)=1$. Then
\begin{equation*}
q_n^h(x) = \left( x - \left( -2h(n)\right) \right) q_{n-1}^h(x) - h(n-1)h(n) q_{n-2}^h(x), \qquad (n \geq 2),
\end{equation*}
with initial values $q_{-1}^{h}\left( x\right) =0$ and
$q_{0}^{h}\left( x\right) =1$.
\end{theorem}

\begin{proof}
With $\alpha _{2}=1$,
$\alpha _{1}=-2$, and $\alpha _{0}=1$ we have
\begin{equation}
\alpha _{2}g\left( n+1
\right) +\alpha _{1}g\left( n
\right) +\alpha _{0}g\left( n-1\right) =0
\label{eq:alphas}
\end{equation}
for all
$n\geq 2$, see (\cite{HNT20}, Example 2.7).
With $\alpha _{2} \, g\left( 1\right) =1$
and $\alpha _{1}\, g\left( 1\right) +\alpha _{2}g\left( 2\right) =-2+4=2$.
From (\cite{HNT20}, Theorem 2.1)
we obtain
\begin{eqnarray}
\nonumber \sum _{m=0}^{2}\left( \alpha _{m}
\frac{h\left( n-1+m\right) }{h\left( n+1\right) }-\frac{x}{h\left( n+1\right) }
\sum _{k=1}^{2-m}\alpha _{m+k}g\left( k\right) \right) P_{n-1+m}^{g,h}\left( x\right) &=& \\
\label{orthogonal}
\frac{h\left( n-1\right) }{h\left( n+1
\right) }P_{n-1}^{g,h} \left( x\right)
-\left(
\frac{2h\left( n
\right) +x }{h\left( n+1
\right) }\right) 
P_{n
}^{g,h}\left( x\right) +P_{n+1
}^{g,h}\left( x\right) &=&0
\end{eqnarray}
for $n\geq
2$ with the initial values
$P_{1}^{g,h}\left( x\right) =
x
$ and
$P_{2}^{g,h}\left( x\right) =\frac{x}{h\left( 2\right) }\left(
x
+2\right) $.
Note that the right hand side is $0$ as in (\ref{eq:alphas}).
Therefore, we obtain
$$
q_{n}^{h}\left( x\right) =\left( x+2h\left( n\right) \right) 
q_{n-1}^{h}\left( x\right) -h\left( n-1\right) h\left( n\right) q_{n-2}^{h}\left( x\right)
$$
for $n\geq 2$. For $n=1$ we obtain from the initial value
$P_{2}^{g
,h}\left( x\right)
$ that
$q_{1}^{h}\left( x\right) =x+2
$.
\end{proof}

We can also give
direct proof of (\ref{orthogonal}) illustrating
the
general idea
of the proof of Theorem~2.1 of \cite{HNT20}
\begin{eqnarray*}
h\left( n+1\right) P_{n+1}^{g,h}\left( x\right) 
&=&x\sum _{k=1}^{n+1}kP_{n+1-k}^{g,h}\left( x\right) \\
&=&
x\sum _{k=1}^{n+1}P_{n+1-k}^{g,h}\left( x\right) +
x\sum _{k=2}^{n+1}\left( k-1\right) P_{n+1-k}^{g,h}\left( x\right) \\
&=&x\sum _{k=1}^{n+1}P_{n+1-k}^{g,h}\left( x\right) +h\left( n\right) P_{n}^{g,h}\left( x\right) \\
&=&xP_{n}^{g,h}\left( x\right) +2h\left( n\right) P_{n}^{g,h}\left( x\right) -h\left( n-1\right) P_{n-1}^{g,h}\left( x\right) .
\end{eqnarray*}
The polynomials $q_n^h(x)$ satisfy a three
term recurrence relation (\ref{3term}) 
with $c_n= -2h(n)$ and $\lambda_n= h(n-1)h(n)$.
\begin{corollary}
The moment functional $\Lambda $
associated with $\{q_n^h\}$ is 
positive definite if and only if $h$ is a real valued positive
function.
\end{corollary}
\begin{examples}
\
\begin{itemize} 
\item Let $h(n)= 1(n)=1$. Then $\{q_n^h\}$ is an OPS with positive definite moment functional: $c_n= -2$ and $\lambda_n=1$.
We obtain $q_n^h(x)= U_n(x/2 +1)$.
\item
Let $h(n)= g(n)=n$. Then $\{q_n^h\}$ is an OPS with positive definite moment functional: $c_n= -2n$ and $\lambda_n=(n-1)n$.
We obtain $q_n^h(x)= n! L_n^{(1)}(-x)$.
\item
Let $h(n)= (-1)^{n+1}$. Then $\{q_n^h\}$ is an OPS with quasi-definite moment functional: $c_n= 2 (-1)^n$ and $\lambda_n=-1$.
\end{itemize}
\end{examples}

%%%%%%%%%%%%%%%%%%%%%%%%%%%%%%%%%%%%%%%%%%%%%%%%%%%%%%%%%%%%%%%%%%%%%%%%%%%%%%%%%%%%%%%%%
%%%%%%%%%%%%%%%%%%%%%%%%%%%%%%%%%%%%%%%%%%%%%%%%%%%%%%%%%%%%%%%%%%%%%%%%%%%%%%%%%%%%%%%%%
\section{Proof of Theorem \ref{suppose}}
Let $h$ be a real-valued arithmetic function, positive and normalized. 
Then the polynomials $\{P_n^{\func{id},h}(x)\}_n$ satisfy
for $n \geq 2
$ the three term recurrence relation
\begin{equation}\label{h:three}
h(n+1) \, P_{n+1}^{\func{id},h}(x) =  ( 2 \, h(n) + x) \, P_n^{\func{id},h}(x) -
h(n-1)
\, P_{n-1}^{\func{id},h}(x),
\end{equation}
with the initial values %$P_0^{\func{id},h}(x)=1$,
$P_1^{\func{id},h}(x)=x$,
$P_{2}^{\func{id},h}\left( x\right) =\left( x+2\right) x/h\left( 2\right) $
and the extension $h(0)=0$.
Note that for $x>0$, (\ref{h:three}) is equivalent to
\begin{equation*}\label{identity}
\frac{P_{n+1}^{\func{id},h}(x)}{P_n^{\func{id},h}(x)} =  \frac{x + 2 \, h(n)}{h(n+1)} - \,  
\frac{h\left( n-1\right) }{h(n+1)} \, \frac{P_{n-1}^{\func{id},h}(x)}{P_{n}^{\func{id},h}(x)}.
\end{equation*}
\begin{remarks}
Let $P_n^{\func{id},h}(x)= 0$ for $x \in \mathbb{R}$, then $T^h_n(x)  \geq 0$. Numerical investigations suggest that
$T_n^h(x) \geq 0$ for all $x \in \mathbb{R}$. In this paper we focus on $x \geq 0$.
\end{remarks}

%%

%%%%%%%%%%%%%%%%%%%%%%%%%%%%%%%%%%%%%%%%%%%%%%%%%%%%%%%%%%%%%%%%%%%%%%%%%%%%%%
%%%%%%%%%%%%%%%%%%%%%%%%%%%%%%%%%%%%%%%%%%%%%%%%%%%%%%%%%%%%%%%%%%%%%%%%%%%%%%

%%%%%%%%%%%%%%%%%%%%%%%%%%%%%%%%%%%%%%%%%%%%%%%%%%%%%%%%%%%%%%%%%%%%%%%%%%%%%%%%%%%%%%%%
%%%%%%%%%%%%%%%%%%%%%%%%%%%%%%%%%%%%%%%%%%%%%%%%%%%%%%%%%%%%%%%%%%%%%%%%%%%%%%%%%%%%%%%%
%%
%%
\begin{proof}[Proof of Theorem \ref{suppose}]
Let $x=0$ then (\ref{Turan})
holds
true. Now let $x >0$. 
The Tur\'an inequality $T_n^h(x) \geq 0$ holds
true if and only if
\begin{equation*} \label{q:general}
\left(\frac{P_n^{\func{id},h}(x)}{P_{n-1}^{\func{id},h}(x)} \right)^2 - \frac{x + 2 \, h(n) }{h(n+1)} \, 
\frac{P_n^{\func{id},h}(x)}{P_{n-1}^{\func{id},h}(x)} + \frac{h(n-1)}{h(n+1)} \, \geq \, 0.
\end{equation*}
We study the roots
$v_n^h(x)$ of the equation:
\begin{equation*}\label{eq:general}
v_n^h(x)^2 - \frac{x + 2 h(n)}{h(n+1)} \, v_n^h(x) + \frac{h(n-1)}{h(n+1)} = 0.
\end{equation*}
The two real solutions are
\begin{eqnarray*}
v_{n,1}^h(x) & = & 
\frac{x + 2 h(n)  - \sqrt{x^2 + 4 \, h(n) \, x + 4( h(n)^2 - h(n-1) \, h(n+1))}}
{2h(n+1)},\\
v_{n,2}^h(x) & = & \frac{x + 2 h(n)  + \sqrt{x^2 + 4 \, h(n) \, x + 4( h(n)^2 - h(n-1) \, h(n+1))}}
{2h(n+1)},
\end{eqnarray*}
where $v_{n,2}^h(x) \geq v_{n,1}^h(x)$. It remains to show that
\begin{equation} \label{vp}
  v_{n,2}^h(x) \leq \frac{P_n^{\func{id},h}(x)}{P_{n-1}^{\func{id},h}(x) }.
\end{equation}
We prove (\ref{vp}) by induction. Let $n=2$. Then 
\begin{equation*}
\frac{P_2^{\func{id},h}(x)}{P_1^{\func{id},h}(x)} =   \frac{x +2}{h(2)} = v_{1,2}^h(x) \geq v_{2,2}^h(x).
\end{equation*}
Now assume that (\ref{vp})
holds
true for $n \geq 2$. We show that
(\ref{vp}) is then valid for $n+1$.
By (\ref{decreasing}) we have $v_{n+1,2}^h(x) \leq v_{n,2}^h(x)$. Thus, 
\begin{eqnarray*}
v_{n+1,2}^h(x) & \leq &  \frac{x + 2 h(n)}{h\left( n+1\right) } - \frac{h(n-1)}{h(n+1)} (v_{n,2}^h(x))^{-1} 
\\
& \leq & \frac{x+ 2 h(n)}{h(n+1)} - 
\frac{h(n-1)}{h(n+1)} \frac{P_{n-1}^{\func{id},h}\left( x\right) }{P_n^{\func{id},h}(x)}= \frac{P_{n+1}^{\func{id},h}(x)}{P_n(x)}.
\end{eqnarray*}
This proves (\ref{vp}) and finally the theorem.
\end{proof}
Finally we add a criterion,
basically reducing (\ref{decreasing}) to the case $x=0$.
\begin{lemma} 
Let $\{h(n)\}_{n \geq 0}$ with $h(0)=
0$ and $h(1)=1$. Suppose
\begin{equation*}\label{start}
v_{n+1,2}^{h}\left( 0\right) \leq v_{n,2}^{h}\left( 0\right). 
\end{equation*}
Let $h(n+1) \geq h(n) \geq 1$,
$\Delta_h(n) \geq \Delta_h(n+1) \geq 0$, and
$$\left( h\left( n+1\right) \right) ^{3} - h\left( n\right) \, \left( h\left( n+2\right) \right) ^{2} \leq 0$$
for all $n \geq 1$.
Then we have for all $0 \leq x_1 \leq x_2$:
\begin{equation*}
v_{n+1,2}^h (x_1) \leq v_{n,2}^h(x_1) \Longrightarrow v_{n+1,2}^h (x_2) \leq v_{n,2}^h(x_2). 
\end{equation*}
\end{lemma}
\begin{proof}
We have for all $n\geq 1$: 
\begin{equation*}
1 \geq \frac{h(n)}{h(n+1)} \geq \frac{h(n-1)}{h(n)} >0 \text{ and }
h(n+1)^3 - h(n) \, h(n+2)^2 \leq 0.
\end{equation*}
We consider whether $v_{n,2}^h(x) -  v_{n+1,2}^h(x)$ is non-negative,
equivalent to the non-negativity of
\begin{eqnarray*}
& &\left(h(n+2)- h(n+1)\right) \, x - 2 \Delta_h(n+1)  \\
& & + \,\, h(n+2) \sqrt{ x^2 + 4 h(n)x + 4 \Delta_h(n)}\\
& & - \,\, h(n+1) \sqrt{ x^2 + 4 h(n+1) x+ 4 \Delta_h(n+1)}.
\end{eqnarray*}
Since $\left(h(n+2)^a-h(n+1)^a \right)     (x_2 - x_1)  \geq 0$ for $a=1,2$ and
\begin{equation*}
4\left(
h(n+2)^2 \, h(n) - h(n+1)^3 \right)                  (x_2 - x_1) \geq 0,
\end{equation*}
the claim of the lemma follows if
$v_{n+1,2}^{h}\left( 0\right) \leq v_{n,2}^{h}\left( 0\right) $.
\end{proof}
\section{Proof of Theorem \ref{result}}
It remains to find interesting sequences $\{h(n)\}_n$ such that $v_{n,2}^h(0) \geq v_{n+1,2}^h(0)$.
We have
\begin{equation*}
v_{n,2}^h(0) = \frac{h(n) + \sqrt{\Delta _{h}\left( n\right) }}{h(n+1)}.
\end{equation*}
This implies that $v_{n+1,2}^h(0) - v_{n,2}^h(0) \leq 0$ if and only if
\begin{equation*}
\label{fundamental}
h(n+2) \, \left( h(n) + \sqrt{\Delta_h(n)} \right) \geq h(n+1) \left( h(n+1) + \sqrt{\Delta _{h}\left( n+1\right) }\right).
\end{equation*}
This leads to
an inequality problem. Let $h: \mathbb{N} \longrightarrow \mathbb{R}_{>0}$ with $h(1)=1$ and
$h$ log-concave. Let 
\begin{equation*}
X(n) := \frac{h(n-1) \, h(n+1)}{h(n)^2}.
\end{equation*}
We are interested in solutions for the following inequality
\begin{equation*}\label{delta}
D(n):= 1 + \sqrt{1- X(n+1)} - X(n+1) \left[ 1 + \sqrt{1-X(n)} \right] \leq 0,
\end{equation*}
which is equivalent to
$v_{n+1,2}^h(0) - v_{n,2}^h(0) \leq 0$.
Let $h(n)=1$ or $h(n)=n$, then
$v_{n,2}^h(0)= 1$ for $n >1$ and $2/h(2)$ for $n=1$. Since
$x U_{n-1}(x/2 + 1) = P_n^{h}(x)$ with $h(n)=1$ and $L_n^{(-1)}(x) = P_n^h(x)$ with $h(n)=n$
we obtain:
\begin{corollary}
The Chebyshev polynomials
of the second kind $U_n(x)$ and the associated
Laguerre polynomials $L_{n}^{\left( -1\right) }\left( x\right) $
satisfy Tur\'an inequalities for all $n \geq 1$ and $x \geq 0$.
\end{corollary}
Both results had been known before. Chebyshev polynomials of the second kind
satisfy the special identity
\begin{equation*}
U_n(x)^2 - U_{n-1}(x) \, U_{n+1}(x)=1.
\end{equation*}
In the paper \cite{Sze48} Tur\'an inequalities for $L_n^{(\alpha)}(x) \, / \, L_n^{(\alpha)}(0)$
for $\alpha >-1$ are obtained. See also the proof of Simic in the context of
Appell polynomials \cite{Si06}.
%%
%%%%%%%%%%%%%%%%%%%%%%%%%%%%%%%%%%%%%%%%%%%%%%%%%%%%%
%%
\subsection{Main lemma}

\begin{lemma}
Let $h(n):= n^s$ for $n \in \mathbb{N}$ and $0
<s<
1$. Let $h(0):=0$.
Then $v_{n,2}^{h}\left( x\right) $
is monotonously decreasing as a function of $n
$.
\end{lemma}

\begin{proof}
Obviously $v_{2,2}^{h}\left( 0\right) \leq v_{1
,2}^{h}\left( 0\right) $. Let now $n\geq 2$. We have to show
\begin{eqnarray*}
D\left( n
\right) &=&1+\sqrt{1-\left( 1-\left( n+1\right) ^{-2}
\right) ^{s}}-\left( 1-\left( n+1\right) ^{-2}
\right) ^{s}\left( 1+\sqrt{1-\left( 1-n^{-2}
\right) ^{s}}\right) \\
&<&0
.
\end{eqnarray*}
This is equivalent to
\begin{eqnarray*}
&&\left( 1-\frac{1}{\left( n+1\right) ^{
2}}
\right) ^{-s}
+\sqrt{\left( 1-\frac{1}{\left( n+1\right) ^{
2}}
\right) ^{-2s}-\left( 1-\frac{1}{\left( n+1\right) ^{
2}}
\right) ^{-s}}\\
&<&1+\sqrt{1-\left( 1-\frac{1}{n^{
2}}
\right) ^{s}}
.
\end{eqnarray*}
Let $S=\left( 1-\left( n+1\right) ^{-2}
\right) ^{-s}$. Then $0<s<1$ is equivalent to
\[
1<S<
\left( 1-\left( n+1\right) ^{-2}\right) ^{-1}
\]
and we
have to show
$
S
+\sqrt{S^{2}-S}
-1<\sqrt{1-S^{-\ln \left( 1-n^{-2}
\right) /\ln \left( 1-\left( n+1\right) ^{-2}
\right) }}
$.
This is equivalent to
\begin{equation}
\left( S
+\sqrt{S^{2}-S}
-1\right) ^{2}<1-S^{-\ln \left( 1-n^{-2}
\right) /\ln \left( 1-\left( n+1\right) ^{-2}
\right) }
.
\label{eq:d2}
\end{equation}
For $S=1$ and $S=\left( 1-\left( n+1\right) ^{-2}\right) ^{-1}
$ both sides equal $0$ and $
n^{-2}$, resp. The second
derivative
of the right hand side is
\[
-\left( -\frac{\ln \left( 1-n^{-2}
\right) }{
\ln \left( 1-\left( n+1\right) ^{-2}
\right) }\right) \left( -\frac{\ln \left( 1-n^{-2}\right) }{
\ln \left(
1-\left( n+1\right) ^{-2}\right) }-1\right) S^{-\frac{\ln \left( 1-n^{-2}
\right) }{
\ln \left( 1-\left( n+1\right) ^{-2}
\right) }-2}<0
.
\]
The first and second derivatives of the left hand side are
\[
2\left( S+\sqrt{S^{2}-S}-1\right) \left( 1+
\frac{2S-1}{
2\left( S^{2}-S\right) ^{
1/2}}
\right)
\]
and
\begin{eqnarray*}
&&2\left( 1+\frac{2S-1}{
2\left( S^{2}-S\right) ^{
\frac{1}{
2}}}
\right) ^{2}
+2\left( S+\sqrt{S^{2}-S}-1\right) \left(
\frac{1}{\left( S^{2}-S\right) ^{
\frac{1}{
2}}}
-
\frac{\left( 2S-1\right) ^{2}}{4\left( S^{2}-S\right) ^{
\frac{3}{
2}}}\right) \\
&=&
\frac{\left( 2\sqrt{
S^{2}-S
}+
2S-1\right) ^{2}\sqrt{S^{2}-S}
-
S-\sqrt{S^{2}-S}+1}{2
\left( S^{2}-S\right) ^{
3/2}}\\
&=&\frac{8\left(
S^{2}-
S
\right) \sqrt{S^{2}-S}+4\left( 2S-1\right) \left( S^{2}-S\right)
-S
+1}{2\left( S^{2}-S\right) ^{3/2}}\\
&=&\frac{8
S
\sqrt{S^{2}-S}+4\left( 2S-1\right)
S
-1
}{2\left( S^{2}-S\right) ^{1
/2}S}>\frac{8\left( S
-1/4\right) ^{2}-3/2}{2S\sqrt{S^{2}-S}}>\frac{3}{2S\sqrt{S^{2}-S}}>0.
\end{eqnarray*}
Therefore, we have shown that the boundary values of (\ref{eq:d2}) agree for
$s\in \left\{ 0,1\right\} $ and for $0<s<1$ the left hand side of (\ref{eq:d2})
has positive second derivative and the right hand side has negative second
derivative. But this implies that the inequality (\ref{eq:d2}) holds true.
\end{proof}

%%%%%%%%%%%%%%%%%%%%%%%%%%%%%%%%%%%%%%%%%%%%%%%%%%%   END
%{\bf Acknowledgments.}

%%%%%%%%%%%%%%%%%%%%%%
\end{document}